\documentclass[11pt,english]{amsart}
\usepackage{ae,aecompl}

\usepackage[T1]{fontenc}
\usepackage[latin9]{inputenc}
\usepackage{geometry}
\usepackage{color}
\usepackage{babel}
\usepackage{enumitem}
\usepackage{amstext}
\usepackage{amsthm}
\usepackage{amssymb}
\usepackage[unicode=true,pdfusetitle,
 bookmarks=true,bookmarksnumbered=false,bookmarksopen=false,
 breaklinks=false,pdfborder={0 0 1},backref=false,colorlinks=true]
 {hyperref}

\makeatletter
\numberwithin{equation}{section}
\numberwithin{figure}{section}
\theoremstyle{plain}
\newtheorem{thmB}{\protect\theorembname}
\theoremstyle{plain}
\newtheorem{thm}{\protect\theoremname}[section]
\theoremstyle{plain}
\newtheorem{lem}[thm]{\protect\lemmaname}
\theoremstyle{remark}
\newtheorem{rem}[thm]{\protect\remarkname}

\usepackage{etex}
\usepackage{amsfonts}
\usepackage{amsthm}
\usepackage{array}
\usepackage{amsmath}
\usepackage{float}

\date{}
\makeatletter
\makeatletter 
\def\clearpage{% 
  \ifvmode 
    \ifnum \@dbltopnum =\m@ne 
      \ifdim \pagetotal <\topskip 
        \hbox{} 
      \fi 
    \fi 
  \fi 
  \newpage 
  \thispagestyle{empty} 
  \write\m@ne{} 
  \vbox{} 
  \penalty -\@Mi 
} 
\makeatother

\newcommand{\Rolosaysnopage}[1]{{ }}

\usepackage{pgf,tikz}
\usetikzlibrary{arrows}

\usepackage{enumitem}
\setenumerate[1]{label=\arabic*. }

\makeatother

\providecommand{\lemmaname}{Lemma}
\providecommand{\remarkname}{Remark}
\providecommand{\theorembname}{Theorem}
\providecommand{\theoremname}{Theorem}

\begin{document}
\global\long\def\essssup{\text{ess sup}}

\title{A sparse approach to mixed weak type inequalities}

\author{Marcela Caldarelli}

\address{(Marcela Caldarelli) Departamento de Matemática, Universidad Nacional
del Sur. Alem 1253, Bahía Blanca, Argentina.}

\email{marcela.caldarelli@uns.edu.ar}

\author{Israel P. Rivera-Ríos}

\address{(Israel P. Rivera-Ríos) CONICET - INMABB, Departamento de Matemática,
Universidad Nacional del Sur. Alem 1253, Bahía Blanca, Argentina.}

\email{israel.rivera@uns.edu.ar}

\thanks{The second author is supported by CONICET PIP 11220130100329CO}
\begin{abstract}
In this paper we provide some quantitative mixed-type estimates assuming
conditions that imply that $uv\in A_{\infty}$ for Calderón-Zygmund
operators, rough singular integrals and commutators. The main novelty
of this paper lies in the fact that we rely upon sparse domination
results, pushing an approach to endpoint estimates that was introduced
in \cite{DSLR} and extended in \cite{LORRArxiv} and \cite{LiPRR}.
\end{abstract}

\maketitle

\section{Introduction and Main Results}

In \cite{MW}, Muckenhoupt and Wheeden introduced a new type of weak
type inequality, that we call mixed type inequality, that consists
in considering a perturbation of the Hardy-Littlewood maximal operator
with an $A_{p}$ weight. Their result was the following
\begin{thmB}
Let $w\in A_{1}$ then 
\[
\left|\left\{ x\in\mathbb{R}\,:\,w(x)Mf(x)>t\right\} \right|\leq\frac{1}{t}\int_{\mathbb{R}}|f|w(x)dx.
\]
\end{thmB}

Although this kind of estimate may seem not very different to the
standard one, the perturbation caused by having the weight inside
the level set makes it way harder to be settled, in contrast with
analogous case of strong type estimates. Furthermore, $w\in A_{1}$
is no longer a necessary condition for this endpoint estimate to hold
(see \cite[Section 5]{MW}).

Later on, Sawyer \cite{S}, motivated by the possibility of providing
a new proof for the Muckenhoupt's theorem, obtained the following
result.
\begin{thmB}
Let $u,v\in A_{1}$ then 
\begin{equation}
uv\left(\left\{ x\in\mathbb{R}\,:\,\frac{M(fv)(x)}{v(x)}>t\right\} \right)\leq\frac{1}{t}\int_{\mathbb{R}}|f|u(x)v(x)dx.\label{eq:SMax}
\end{equation}
\end{thmB}

Sawyer also conjectured that \eqref{eq:SMax} should hold as well
for the Hilbert transform. Cruz-Uribe, Martell and Pérez \cite{CUMP}
generalized \eqref{eq:SMax} to higher dimensions and actually proved
that Sawyer's conjecture holds for Calderón-Zygmund operators via
the following extrapolation argument.
\begin{thmB}
Assume that for every $w\in A_{\infty}$ and some $0<p<\infty$, 
\[
\|Tf\|_{L^{p}(w)}\leq c_{w}\|Gf\|_{L^{p}(w)}.
\]
Then for every $u\in A_{1}$ and every $v\in A_{\infty}$
\[
\|Tf\|_{L^{1,\infty}(uv)}\leq\|Gf\|_{L^{1,\infty}(uv)}.
\]
\end{thmB}

The conditions on the weights in that extrapolation result lead them
to conjecture that \eqref{eq:SMax}, and consequently the corresponding
estimate for Calderón-Zygmund operators should hold as well with $u\in A_{1}$
and $v\in A_{\infty}$. That conjecture was positively answered recently
in \cite{LiOP} where several quantitative estimates were provided
as well. At this point we would like to mention, as well, a recent
generalization provided for Orlicz maximal operators in \cite{B}.

In \cite{CUMP}, besides the aforementioned results, it was shown
that \eqref{eq:SMax} holds if $u\in A_{1}$ and $v\in A_{\infty}(u)$
(see Section \ref{sec:ApWeightsOrliczmaximal} for the precise definition
of $A_{p}(u)$). The advantage of that condition is that the product
$uv$ is an $A_{\infty}$ weight. Over the past few years, there have
been new contributions under those assumptions such as \cite{BCP2}
for the case of fractional integrals and related operators, \cite{OP,OPR}
for related quantitative estimates and \cite{LiOPi} for multilinear
extensions. 

The case of commutators of Calderón-Zygmund operators was settled
in \cite{BCP}. Recall that given $T$ a Calderón-Zygmund operator,
$b\in Osc_{\exp L^{r}}\subset BMO$ (see Section \ref{sec:ApWeightsOrliczmaximal}
for the precise definition) and a positive integer $m$, we define
the higher order commutator $T_{b}^{m}f$ by
\[
T_{b}^{m}f(x)=b(x)T_{b}^{m-1}f(x)-T_{b}^{m-1}(bf)(x)
\]
where $T_{b}^{1}f(x)=b(x)Tf(x)-T(bf)(x)$.

Now we turn to our contribution. Our approach exploits sparse domination
and ideas from \cite{LiPRR} that can be traced back to \cite{DSLR}.
In the case of commutators our approach is inspired by \cite{LORRArxiv}
as well. The main novelty of our proofs is precisely that, in contrast
with the techniques used up until now to deal with this kind of questions,
we heavily rely upon sparse domination. Our first result is the following.
\begin{thm}
\label{Thm:1}Let $u\in A_{1}$ and $v\in A_{p}(u)$ for some $1<p<\infty$.
\begin{enumerate}
\item If $T$ is a Calderón-Zygmund operator, 
\[
\left\Vert \frac{T(fv)(x)}{v(x)}\right\Vert _{L^{1,\infty}(uv)}\leq c_{n,p}[uv]_{A_{\infty}}[u]_{A_{1}}\log\left(e+[uv]_{A_{\infty}}[u]_{A_{1}}[v]_{A_{p}(u)}\right)\|f\|_{L^{1}(uv)}
\]
and if $m$ is a positive integer, $r>1$ and $b\in Osc_{\exp L^{r}}$
then 
\begin{equation}
uv\left(\left\{ x\in\mathbb{R}^{n}\,:\,\left|\frac{T_{b}^{m}(fv)}{v}\right|>t\right\} \right)\leq c_{n,p}\Gamma_{u,v}^{m}\int_{\mathbb{R}^{n}}\Phi_{\frac{m}{r}}\left(\frac{|f|\|b\|_{Osc_{\exp L^{r}}}^{m}}{t}\right)uv\label{eq:DepComm}
\end{equation}
where 
\[
\Gamma_{u,v}^{m}=\sum_{h=0}^{m}[u]_{A_{1}}[uv]_{A_{\infty}}^{1+\frac{h}{r}}[u]_{A_{\infty}}^{\frac{m-h}{r}}\log\left(e+[u]_{A_{1}}[uv]_{A_{\infty}}^{1+\frac{h}{r}}[u]_{A_{\infty}}^{\frac{m-h}{r}}[v]_{A_{p}(u)}\right)^{1+\frac{h}{r}}
\]
and $\Phi_{\rho}(t)=t\left(1+\log^{+}(t)\right)^{\rho}.$
\item If $\Omega\in L^{\infty}(\mathbb{S}^{n-1})$ then
\[
\left\Vert \frac{T_{\Omega}(fv)(x)}{v(x)}\right\Vert _{L^{1,\infty}(uv)}\leq c_{n,p}[uv]_{A_{\infty}}[u]_{A_{1}}[u]_{A_{\infty}}\log\left(e+[uv]_{A_{\infty}}[u]_{A_{1}}[u]_{A_{\infty}}[v]_{A_{p}(u)}\right)\|f\|_{L^{1}(uv)}.
\]
\[
\]
\end{enumerate}
\end{thm}

We would like to note that in the case $u=1$, in the case of Calderón-Zygmund
operators, the estimate above reduces to 
\[
\left\Vert \frac{T(fv)(x)}{v(x)}\right\Vert _{L^{1,\infty}(v)}\leq c_{n,p}[v]_{A_{\infty}}\log\left(e+[v]_{A_{p}}\right)\|f\|_{L^{1}(v)}\qquad p\geq1.
\]
That estimate improves the bound provided in \cite[Theorems 1.16 and 1.17]{OPR},
namely, 
\[
\left\Vert \frac{T(fv)(x)}{v(x)}\right\Vert _{L^{1,\infty}(v)}\leq c_{n,p}[v]_{A_{p}}\log\left(e+[v]_{A_{p}}\right)\|f\|_{L^{1}(v)}\qquad p\geq1.
\]

In the case of the commutator our approach provides a new proof of
\cite[Theorem 2]{BCP} obtaining a quantitative estimate as well.
An arguable drawback of the estimates above is that in neither of
them we recover the best known dependence in the case $v=1$. We wonder
whether the factor $[uv]_{A_{\infty}}$ in each of them can be removed.

In our following result we assume that $v\in A_{1}$ and $u\in A_{1}(v)$.
It is not hard to check that those conditions are equivalent to assume
that $u\in A_{1}$ and $v\in A_{1}(u)$, so there is no gain in terms
of the size of the class of weights considered. However, in this case,
if $v=1$ we recover the best known estimates for $u\in A_{1}$.
\begin{thm}
\label{Thm:2}Let $v\in A_{1}$ and $u\in A_{1}(v)$.
\begin{enumerate}
\item If $T$ is a Calderón-Zygmund operator

\[
\left\Vert \frac{T(fv)(x)}{v(x)}\right\Vert _{L^{1,\infty}(uv)}\leq c_{n,T}[v]_{A_{1}}[v]_{A_{\infty}}[u]_{A_{1}(v)}\log\left(e+[uv]_{A_{\infty}}[v]_{A_{1}}\right)\|f\|_{L^{1}(uv)}
\]
and if $m$ is a positive integer, $r>1$ and $b\in Osc_{\exp L^{r}}$
then 
\begin{equation}
uv\left(\left\{ x\in\mathbb{R}^{n}\,:\,\left|\frac{T_{b}^{m}(fv)}{v}\right|>t\right\} \right)\leq c_{n,p}\Gamma_{u,v}^{m}\int_{\mathbb{R}^{n}}\Phi_{\frac{m}{r}}\left(\frac{|f|\|b\|_{Osc_{\exp L^{r}}}^{m}}{t}\right)uv\label{eq:DepComm-1}
\end{equation}
where 
\[
\Gamma_{u,v}^{m}=\sum_{h=0}^{m}[v]_{A_{1}}[v]_{A_{\infty}}^{\frac{h}{r}}[uv]_{A_{\infty}}^{\frac{m-h}{r}}[u]_{A_{1}(v)}[v]_{A_{\infty}}\log\left(e+[v]_{A_{1}}[v]_{A_{\infty}}^{\frac{h}{r}}[uv]_{A_{\infty}}^{\frac{m-h}{r}}\right)^{1+\frac{h}{r}}
\]
and $\Phi_{\rho}(t)=t\left(1+\log^{+}(t)\right)^{\rho}.$
\item If $\Omega\in L^{\infty}(\mathbb{S}^{n-1})$ then
\[
\left\Vert \frac{T_{\Omega}(fv)(x)}{v(x)}\right\Vert _{L^{1,\infty}(uv)}\leq c_{n,\Omega}[uv]_{A_{\infty}}[v]_{A_{1}}[u]_{A_{1}(v)}[v]_{A_{\infty}}\log\left(e+[uv]_{A_{\infty}}[v]_{A_{1}}\right)\|f\|_{L^{1}(uv)}
\]
\[
\]
\end{enumerate}
\end{thm}

As we pointed out above, notice that this result recovers the best
dependences known obtained in \cite{LOP,LOP2,LiPRR,LORRArxiv,IFRR}
in the case, $v=1$. Furthermore, in case of the commutator we obtain
the following estimate
\begin{equation}
u\left(\left\{ x\in\mathbb{R}^{n}\,:\,\left|T_{b}^{m}(f)\right|>t\right\} \right)\leq c_{n,p}[u]_{A_{1}}[u]_{A_{\infty}}^{\frac{m}{r}}\log\left(e+[u]_{A_{\infty}}\right)\int_{\mathbb{R}^{n}}\Phi_{\frac{m}{r}}\left(\frac{|f|\|b\|_{Osc_{\exp L^{r}}}^{m}}{t}\right)u.\label{eq:endpointOsc}
\end{equation}
Observe that \eqref{eq:endpointOsc} contains as a particular case
the endpoint estimate obtained in \cite{IFRR} and provides precise
quantitative bound for the case in which the symbol has better local
decay properties than $BMO$ functions. We recall that in \cite{A},
it was shown that if a commutator of a certain singular integral satisfies
a weak-type $(1,1)$ estimate then $b\in L^{\infty}$ and that the
$L\log L$ estimate, first settled in \cite{P}, implies that $b\in BMO$.
Bearing those results in mind we wonder whether $b\in Osc_{\exp L^{r}}$
should be a neccesary condition for \eqref{eq:endpointOsc}, at least
in the case $u=1$, to hold.

The rest of the paper is organized as follows. Section \ref{sec:Preliminaries}
is devoted to provide some basic results and to fix notation that
will be used throughout the remainder of the paper and in Section
\ref{sec:Proofs} we provide the proofs of the main results.

\section{\label{sec:Preliminaries}Preliminaries}

\subsection{Sparse domination results}

In this section we begin borrowing some definitions from \cite{LN}. 

Given a cube $Q$ we denote by $\mathcal{D}(Q)$ the standard dyadic
grid relative to $Q$. 

We say that a family of cubes $\mathcal{D}$ is a dyadic lattice it
satisfies the following conditions.
\begin{enumerate}
\item If $Q\in\mathcal{D}$ then $\mathcal{D}(Q)\subset\mathcal{D}$.
\item If $P,Q\in\mathcal{D}$ then there exists $R\in\mathcal{D}$ such
that $P,Q\in\mathcal{D}(R)$.
\item For every compact set $K\subset\mathbb{R}^{d}$ there exists some
$Q\in\mathcal{D}$ such that $K\subset Q$.
\end{enumerate}
We recall that $\mathcal{S}$ is a $\eta$-sparse family if for every
$Q\in\mathcal{S}$ there exists $E_{Q}\subset Q$ such that
\begin{enumerate}
\item $\eta|Q|\leq|E_{Q}|$. 
\item The sets $E_{Q}$ are pairwise disjoint.
\end{enumerate}
In some situations it is useful to approximate arbitrary cubes by
dyadic cubes. For that purpose, one dyadic lattice is not enough,
however $3^{n}$ are. That fact follows from the following Lemma that
we borrow from \cite{LN}.
\begin{lem}
\label{Lem:3ndlt}For every dyadic lattice $\mathcal{D}$ there exist
$3^{n}$ dyadic lattices $\mathcal{D}_{j}$ such that 
\[
\left\{ 3Q\,:\,Q\in\mathcal{D}\right\} =\bigcup_{j=1}^{3^{n}}\mathcal{D}_{j}
\]
and for every cube $Q\in\mathcal{D}$ and $j=1,\dots,3^{n}$, there
exists a unique cube $R\in\mathcal{D}_{j}$ of sidelenght $l_{R}=3l_{Q}$
containing $Q$.
\end{lem}

In the last years, and after Lerner's simplification the proof of
the $A_{2}$ theorem \cite{L} that had been settled earlier by Hytönen
\cite{H}, the sparse domination approach has been widely and succesfully
applied in the theory of weights. The philosophy behind that approach
consists in controlling, in some sense, the operator that we want
to study by suitable sparse operators and providing estimates for
the latter ones, which are in general easier to settle.

In the following Theorem we gather the sparse domination results that
we will rely upon in the main results of the paper.
\begin{thm}
Let $f\in\mathcal{C}_{c}^{\infty}.$ 
\begin{description}
\item [{\cite{CAR,LN,La,HRT,Le}}] If $T$ is a Calderón-Zygmund operator
there exist $3^{n}$ $\varepsilon$-sparse families contained in $3^{n}$
dyadic lattices $\mathcal{D}_{j}$ such that 
\[
|Tf(x)|\leq c_{n,T,\varepsilon}\sum_{j=1}^{3^{n}}A_{\mathcal{S}}\left(|f|\right)(x)
\]
where $A_{\mathcal{S}}f(x)=\sum_{Q\in\mathcal{S}}\frac{1}{|Q|}\int_{Q}f(y)dy\chi_{Q}(x)$. 
\item [{\cite{IFRR,LORRArxiv}}] If $T$ is a Calderón-Zygmund operator
and $b\in BMO$ then there exist $3^{n}$ $\varepsilon$-sparse families
contained in $3^{n}$ dyadic lattices $\mathcal{D}_{j}$ such that
\[
|T_{b}^{m}f(x)|\leq c_{n,T,\varepsilon}\sum_{j=1}^{3^{n}}\sum_{h=0}^{m}A_{\mathcal{S}}^{m,h}(b,f)(x)
\]
where $h=0,\dots,m$ and
\[
A_{\mathcal{S}}^{m,h}(b,f)(x)=\sum_{Q\in\mathcal{S}}|b(x)-b_{Q}|^{m-h}\frac{1}{|Q|}\int_{Q}|b-b_{Q}|^{h}f\chi_{Q}(x).
\]
\item [{\cite{CACDiO,LRough}}] If $\Omega\in L^{\infty}(\mathbb{S}^{n-1})$
then there exists a sparse family $\mathcal{S}$ such that 
\begin{equation}
\left|\int_{\mathbb{R}^{n}}T_{\Omega}fg\right|\leq c_{n,\Omega}r'\Lambda_{\mathcal{S}}^{r}(f,g)\quad r>1\label{eq:Rough}
\end{equation}
where $f\in L^{r}$ and $g\in L_{\text{loc}}^{1}$
\[
\Lambda_{\mathcal{S}}^{r}(f,g)=\sum_{Q\in\mathcal{S}}\frac{1}{|Q|}\int_{Q}|f|\left(\frac{1}{|Q|}\int_{Q}|g|^{r}\right)^{\frac{1}{r}}|Q|.
\]
\end{description}
\end{thm}

\begin{rem}
\label{Rem:3nRough}Notice that the $3^{n}$-dyadic lattices trick
(Lemma \ref{Lem:3ndlt}) allows us to show that for every dyadic lattice
$\mathcal{D}$, 
\[
\Lambda_{\mathcal{S}}^{r}(f,g)\leq\sum_{j=1}^{3^{n}}\Lambda_{\mathcal{S}_{j}}^{r}(f,g)
\]
where each $\mathcal{S}_{j}\subset\mathcal{D}_{j}$ and the choice
of the dyadic lattices $\mathcal{D}_{j}$ is independent of $f,g$. 
\end{rem}

\subsection{$A_{p}$ weights and Orlicz maximal functions\label{sec:ApWeightsOrliczmaximal}}

We recall that given a weight $u$, $v\in A_{p}(u)$ if 
\[
[v]_{A_{p}(u)}=\sup_{Q}\frac{1}{u(Q)}\int_{Q}vu\left(\frac{1}{u(Q)}\int_{Q}v^{-\frac{1}{p-1}}u\right)^{p-1}<\infty
\]
in the case $1<p<\infty$ and 
\[
[v]_{A_{1}(u)}=\left\Vert \frac{M_{u}v}{v}\right\Vert _{L^{\infty}}<\infty
\]
where $M_{u}v=\sup_{Q}\frac{1}{u(Q)}\int_{Q}vu$. Analogously if $u=1$
we recover the classical Muckenhoupt's condition. 

We would like also to recall that
\[
A_{\infty}=\bigcup_{p\geq1}A_{p}.
\]
This class of weights is characterized in terms of the following condition
\[
w\in A_{\infty}\iff[w]_{A_{\infty}}=\sup_{Q}\frac{1}{w(Q)}\int_{Q}M(\chi_{Q}w)<\infty.
\]
This characterization was discovered by Fujii \cite{F} and rediscovered
by Wilson \cite{W}. Up until now that $[w]_{A_{\infty}}$ is the
smallest constant characterizing the $A_{\infty}$ class (see Pérez
and Hytönen \cite{HPAinfty}). A result that we will use as well is
the following reverse Hölder inequality that was obtained in \cite{HPAinfty}
(see \cite{HPR} for another proof).
\begin{lem}
There exists $\tau_{n}$ such that for every $w\in A_{\infty}$ 
\[
\left(\frac{1}{|Q|}\int_{Q}w^{r_{w}}\right)^{\frac{1}{r_{w}}}\leq\frac{2}{|Q|}\int_{Q}w
\]
where $r_{w}=1+\frac{1}{\tau_{n}[w]_{A_{\infty}}}$.
\end{lem}

We recall that given a Young function $A:[0,\infty)\rightarrow[0,\infty)$,
namely a convex, non-decreasing function such that $A(0)=0$ and $\frac{A(t)}{t}\rightarrow\infty$
when $t\rightarrow\infty$ we can define 
\[
\|f\|_{A(u),Q}=\|f\|_{A(L)(u),Q}=\inf\left\{ \lambda>0\,:\,\frac{1}{u(Q)}\int_{Q}A\left(\frac{|f(x)|}{\lambda}\right)u(x)dx\leq1\right\} .
\]
It is possible to provide a definition of the norm equivalent to the
latter (see \cite{KR1961}), namely
\[
\|f\|_{A(u),Q}\simeq\inf_{\mu>0}\left\{ \mu+\frac{\mu}{u(Q)}\int_{Q}A\left(\frac{|f(x)|}{\mu}\right)u(x)dx\right\} .
\]
Associated to each Young $A$ function there exists another Young
function $\overline{A}$ such that 
\[
\frac{1}{u(Q)}\int_{Q}|fg|u\leq2\|f\|_{A(u),Q}\|g\|_{\overline{A}(u),Q}.
\]
We shall drop $u$ in the notation in the case of Lebesgue measure.
Some particular cases of interest for us will be $A(t)=t\log(e+t)^{\frac{1}{r}}$
and $\overline{A}(t)=\exp(t^{r})-1$ for $r>1$. 

Let $u$ a weight and $A$ a Young function. We define the maximal
operator $M_{A(u)}^{\mathcal{F}}$ by
\[
M_{A(u)}^{\mathcal{F}}f(x)=\sup_{x\in Q\in\mathcal{F}}\|f\|_{A(u),Q}.
\]
where the supremum is taken over all the cubes in the family $\mathcal{F}$.
We shall drop the superscript in case the context makes clear the
family of cubes considered. If we choose $A(t)=t$ and $u=1$ and
$\mathcal{F}$ is the family of all cubes we recover the classical
Hardy-Littlewood operator.

Now we recall if $b\in BMO$, then
\[
\sup_{Q}\|b-b_{Q}\|_{\exp L,Q}\leq c_{n}\|b\|_{BMO}.
\]
It is possible to define classes of symbols with even better properties
of integrability than $BMO$ symbols. Given $r>1$ we say that $b\in Osc_{\exp L^{r}}(w)$
if 
\[
\|b\|_{Osc_{\exp L^{r}}(w)}=\sup_{Q}\|b-b_{Q}\|_{\exp L^{r}(w),Q}<\infty.
\]
Note that $Osc_{\exp L^{r}}\subsetneq BMO$ for every $r>1$. It is
not hard to prove that for those classes of functions the following
estimates hold.
\begin{lem}
Let $w\in A_{\infty}$ and $b\in Osc_{\exp L^{r}}$. Then 
\[
\|b-b_{Q}\|_{\exp L^{r}(w)}\leq c[w]_{A_{\infty}}^{\frac{1}{r}}\|b\|_{Osc_{\exp L^{r}}}.
\]
Furthermore, if $j>0$ then
\[
\left\Vert \left|b-b_{Q}\right|^{j}\right\Vert _{\exp L^{\frac{r}{j}}(w)}\leq c[w]_{A_{\infty}}^{\frac{j}{r}}\|b\|_{Osc_{\exp L^{r}}}^{j}.
\]
\end{lem}

We end up this section with a result that allows us to change the
underlying weight of Orlicz averages.
\begin{lem}
\label{Lem:ContAvg}Let $u$ a weight, $v\in A_{p}(u)$, and $\Phi$
a Young function. Then, for every cube $Q$, 
\[
\|f\|_{\Phi(u),Q}\leq\|f\|_{[v]_{A_{p}(u)}\Phi^{p}(L)(uv),Q}.
\]
\end{lem}

We remit to \cite{O,RR} for more information about Young functions
and Orlicz spaces.

\section{\label{sec:Proofs}Proofs of the main results }

\subsection{Scheme of the proofs}

Before we provide the needed lemmata and the proofs of the main results
we would like to briefly outline the scheme that we are going to follow
for each of the proofs of the estimates in the main results that,
as we mentioned in the introduction, can be traced back to \cite{DSLR,LORRArxiv,LiPRR}.
Let $T$ a linear operator, possibly a sparse operator and let $\tilde{M}_{uv}f$
a dyadic, in some sense, maximal operator such that
\[
\]
\[
uv\left(\left\{ x\in\mathbb{R}^{d}\,:\,|\tilde{M}_{uv}f(x)|>t\right\} \right)\leq\frac{1}{t}\int A\left(\frac{|f|}{t}\right)uv
\]
where $A$ is a Young function. First, notice that 
\[
\begin{split}uv\left(\left\{ x\in\mathbb{R}^{d}\,:\,\frac{|T(fv)(x)|}{v}>1\right\} \right) & =uv\left(\left\{ x\in\mathbb{R}^{d}\,:\,\frac{|T(fv)(x)|}{v}>1,\tilde{M}_{uv}f(x)\leq\frac{1}{2}\right\} \right)\\
 & +uv\left(\left\{ x\in\mathbb{R}^{d}\,:\,|\tilde{M}_{uv}f(x)|>\frac{1}{2}\right\} \right).
\end{split}
\]
Since the desired estimate holds for the second term it suffices to
control the first one. Let us call 
\[
G=\left\{ x\in\mathbb{R}^{d}\,:\,\frac{|T(fv)(x)|}{v}>1,\tilde{M}_{uv}f(x)\leq\frac{1}{2}\right\} .
\]
Then it suffices to prove
\begin{equation}
uv\left(\left\{ x\in\mathbb{R}^{d}\,:\,\frac{|T(fv)(x)|}{v}>1,\tilde{M}_{uv}f(x)\leq\frac{1}{2}\right\} \right)\leq c_{n,T}\kappa_{u,v}\int A\left(|f|\right)uv+\frac{1}{2}uv(G).\label{eq:Red}
\end{equation}
This yields 
\[
uv\left(\left\{ x\in\mathbb{R}^{d}\,:\,\frac{|T(fv)(x)|}{v}>1,\tilde{M}_{uv}f(x)\leq\frac{1}{2}\right\} \right)\leq2c_{n,T}\kappa_{u,v}\int A\left(|f|\right)uv
\]
and consequently 
\[
uv\left(\left\{ x\in\mathbb{R}^{d}\,:\,\frac{|T(fv)(x)|}{v}>1\right\} \right)\leq2c_{n,T}\kappa_{u,v}\int A\left(|f|\right)uv
\]
which by homogeneity allows us to end up the proof.

The purpose of the following sections we will be settling \eqref{eq:Red}
for the operators in the main theorems. To achieve in that task we
will rely upon sparse domination results, and more in particular we
will use suitable splittings of the sparse families involved in the
spirit of \cite{DSLR,LORRArxiv,LiPRR}.

\subsection{Lemmatta}

Before starting with the proofs of the main results we provide some
technical lemmas.
\begin{lem}
\label{Lem:DoubleSum}Let $\gamma_{1},\gamma_{2}>1$. For every $j,k$
non negative integers let 
\[
\alpha_{k,j}=\min\{\gamma_{1}2^{-k}j^{\rho_{1}},\beta\gamma_{2}2^{-j}2^{-k}2^{\delta k}k^{\rho_{2}}\},
\]
where $\rho_{1},\rho_{2},\delta\geq0$. Then
\[
\sum_{j,k\geq0}\alpha_{k,j}\leq c_{\rho_{1},\rho_{2},\gamma,\delta}\gamma_{1}\log_{2}\left(e+\gamma_{2}\right)^{1+\rho_{1}}+\frac{1}{2\gamma}\beta,
\]
where $\gamma\geq1$.
\end{lem}

\begin{proof}
We start writing
\[
\sum_{j,k\geq0}\alpha_{k,j}=\sum_{j\geq\left\lceil \log_{2}\left((e+\gamma_{2})8\gamma\right)\right\rceil +\left(\left\lceil \delta+\rho_{2}\right\rceil +1\right)k}\alpha_{k,j}+\sum_{j<\left\lceil \log_{2}\left((e+\gamma_{2})8\gamma\right)\right\rceil +\left(\left\lceil \delta+\rho_{2}\right\rceil +1\right)k}\alpha_{k,j}
\]
For the first term, notice that 
\[
\begin{split} & \sum_{j\geq\left\lceil \log_{2}\left((e+\gamma_{2})\delta\right)\right\rceil +\left(\left\lceil \delta\right\rceil +1\right)k}\alpha_{k,j}\\
 & \leq\beta\gamma_{2}\sum_{k=0}^{\infty}2^{-k}2^{\delta k}k^{\rho_{2}}\sum_{j\geq\left\lceil \log_{2}\left((e+\gamma_{2})8\right)\right\rceil +\left(\left\lceil \delta+\rho_{2}\right\rceil +1\right)k}2^{-j}\\
 & =\beta\gamma_{2}\sum_{k=0}^{\infty}2^{-k}2^{\delta k}k^{\rho_{2}}2^{-\left\lceil \log_{2}\left((e+\gamma_{2})8\right)\right\rceil -\left(\left\lceil \delta+\rho_{2}\right\rceil +1\right)k}\\
 & \leq\frac{\beta\gamma_{2}}{(e+\gamma_{2})8\gamma}\sum_{k=0}^{\infty}2^{-k-\rho_{2}}k^{\rho_{2}}\\
 & \leq\frac{\beta\gamma_{2}}{(e+\gamma_{2})8\gamma}\sum_{k=0}^{\infty}2^{-(1+\rho_{2})k}2^{\rho_{2}\log k}\\
 & \leq\frac{\beta\gamma_{2}}{(e+\gamma_{2})8\gamma}\sum_{k=0}^{\infty}2^{-k}\\
 & \leq\frac{2\gamma_{2}\beta}{(e+\gamma_{2})8\gamma}\\
 & \leq\frac{\gamma_{2}}{(e+\gamma_{2})4\gamma}\beta\leq\frac{1}{2\gamma}\beta
\end{split}
\]
For the second term, we observe that
\[
\begin{split} & \sum_{j<\left\lceil \log_{2}\left((e+\gamma_{2})8\gamma\right)\right\rceil +\left(\left\lceil \delta+\rho_{2}\right\rceil +1\right)k}\alpha_{k,j}\\
 & \leq\gamma_{1}\sum_{k=0}^{\infty}2^{-k}\sum_{1\leq j<\left\lceil \log_{2}\left((e+\gamma_{2})8\gamma\right)\right\rceil +\left(\left\lceil \delta+\rho_{2}\right\rceil +1\right)k}j^{\rho_{1}}\\
 & \leq\gamma_{1}\sum_{k=0}^{\infty}\left(\left\lceil \log_{2}\left((e+\gamma_{2})8\gamma\right)\right\rceil +\left(\left\lceil \delta+\rho_{2}\right\rceil +1\right)k\right)^{1+\rho_{1}}2^{-k}\\
 & \leq c2\left(\delta+\rho_{2}\right)\gamma_{1}\log_{2}\left((e+\gamma_{2})8\gamma\right)^{1+\rho_{1}}\\
 & \leq c_{\rho_{1},\rho_{2},\gamma,\delta}\gamma_{1}\log\left(e+\gamma_{2}\right)^{1+\rho_{1}}
\end{split}
\]
and we are done.
\end{proof}
The second result we will rely upon is the following.
\begin{lem}
\label{Lem:DisEQ}Let $A$ a submultiplicative Young function and
$\mathcal{S}$ a $\frac{A(8)}{1+A(8)}$-sparse family. Let $f\in\mathcal{C}_{c}^{\infty}$
and $w\in A_{\infty}$ and assume that for every $Q\in\mathcal{S}$
\[
2^{-j-1}\leq\langle f\rangle_{A(L)(w)Q}\leq2^{-j}.
\]
Then for every $Q\in\mathcal{S}$ there exists $\tilde{E}_{Q}\subseteq Q$
such that 
\[
\sum_{Q\in\mathcal{S}}\chi_{\tilde{E_{Q}}}(x)\leq c_{n}[w]_{A_{\infty}}
\]
and 
\[
w(Q)\|f\|_{A(w),Q}\leq4\frac{A(2^{j+2})}{2^{j+2}}\int_{\tilde{E_{Q}}}A\left(|f|\right)w.
\]
 
\end{lem}

\begin{proof}
We split the family $\mathcal{S}$ in the following way
\[
\begin{split}\mathcal{S}^{0} & =\{\text{Maximal in }\mathcal{S}\}\\
\mathcal{S}^{1} & =\{\text{Maximal in }\mathcal{S}\setminus\mathcal{S}^{0}\}\\
 & \dots\\
\mathcal{S}^{i} & =\{\text{Maximal in }\mathcal{S}\setminus\cup_{r=0}^{i-1}\mathcal{S}^{r}\}
\end{split}
\]
Note that since $w\in A_{\infty}$ we have that, for each cube $Q$
and each measurable subset $E\subset Q$, 
\[
w(E)\leq2\left(\frac{|E|}{|Q|}\right)^{\frac{1}{c_{n}[w]_{A_{\infty}}}}w(Q)
\]
In particular if $Q\in\mathcal{S}^{i}$ and $J_{1}=\bigcup_{P\in\mathcal{S}^{i+1},\,P\subset Q}P$
then 
\[
|J_{1}|=\left|\bigcup_{P\in\mathcal{S}^{i+1},\,P\subset Q}P\right|\leq\left(\frac{1+A(4)}{A(4)}-1\right)|Q|=\frac{1}{A(4)}|Q|.
\]
And this yields
\[
w(J_{1})\leq\left(\frac{1}{A(8)}\right)^{\frac{1}{c_{n}[w]_{A_{\infty}}}}w(Q).
\]
Furthermore, arguing by induction, if we denote $J_{\nu}=\bigcup_{P\in\mathcal{S}^{i+\nu},\,P\subset Q}P$
\[
w(J_{\nu})\leq\left(\frac{1}{A(8)}\right)^{\frac{\nu}{c_{n}[w]_{A_{\infty}}}}w(Q)
\]
And in particular if we choose $\nu=\left\lceil c_{n}[w]_{A_{\infty}}\right\rceil $,
then 
\[
w(J_{\nu})\leq\frac{1}{A(8)}w(Q)
\]
 Let $Q\in\mathcal{S}^{i}$. and let $\tilde{E_{Q}}=Q\setminus\bigcup_{P\in\mathcal{S}^{i+\left\lceil c_{n}[w]_{A_{\infty}}\right\rceil }}P$.
\[
\begin{split} & w(Q)\|f\|_{A(w),Q}\\
 & \leq w(Q)\left\{ 2^{-j-2}+\frac{2^{-j-2}}{w(Q)}\int_{Q}A\left(2^{j+2}|f|\right)w\right\} \\
 & \leq w(Q)2^{-j-2}+\frac{1}{2^{j+2}}\int_{Q}A\left(2^{j+2}|f|\right)w\\
 & \leq w(Q)2^{-j-1}+\frac{1}{2^{j+2}}\int_{\tilde{E_{Q}}}A\left(2^{j+2}|f|\right)w+\frac{1}{2^{j+2}}\sum_{P\in\mathcal{S}_{j,k}^{i+\left\lceil c_{n}[w]_{A_{\infty}}\right\rceil }}\int_{P}A\left(2^{j+2}|f|\right)w\\
 & \leq w(Q)2^{-j-2}+\frac{A(2^{j+2})}{2^{j+2}}\int_{\tilde{E_{Q}}}A\left(|f|\right)w+\frac{1}{2^{j+2}}\sum_{P\in\mathcal{S}_{j,k}^{i+\left\lceil c_{n}[w]_{A_{\infty}}\right\rceil }}\int_{P}A\left(2^{j+2}|f|\right)w
\end{split}
\]
Observe that we can bound the last term as follows
\[
\begin{split} & \sum_{P\in\mathcal{S}_{j,k}^{i+\left\lceil c_{n}[w]_{A_{\infty}}\right\rceil }}\int_{P}A\left(2^{j+2}|f|\right)w\\
 & \leq A(4)\sum_{P\in\mathcal{S}_{j,k}^{i+\left\lceil c_{n}[w]_{A_{\infty}}\right\rceil }}w(P)\frac{1}{w(P)}\int_{P}A\left(2^{j}|f|\right)w\\
 & \leq A(4)\sum_{P\in\mathcal{S}_{j,k}^{i+\left\lceil c_{n}[uv]_{A_{\infty}}\right\rceil }}w(P)\\
 & \leq\frac{A(4)}{2}\frac{1}{A(8)}w(Q)\leq\frac{1}{4}w(Q)
\end{split}
\]
Hence
\[
\begin{split}w(Q)\|f\|_{A(w),Q} & \leq\frac{1}{2^{j+2}}w(Q)+\frac{A(2^{j+2})}{2^{j+2}}\int_{\tilde{E_{Q}}}A\left(|f|\right)w+\frac{1}{2^{j+2}}\frac{1}{4}w(Q)\\
 & \leq\left(\frac{1}{2}+\frac{1}{4}\right)w(Q)\|f\|_{A(w),Q}+\frac{A(2^{j+2})}{2^{j+2}}\int_{\tilde{E_{Q}}}A\left(|f|\right)w\\
 & =\frac{3}{4}w(Q)\|f\|_{A(w),Q}+\frac{A(2^{j+2})}{2^{j+2}}\int_{\tilde{E_{Q}}}A\left(|f|\right)w,
\end{split}
\]
from which readily follows the desired conclusion. 
\end{proof}
The following lemma will be also used repeatedly.
\begin{lem}
\label{Lem:UnQ}Let $w\in A_{\infty}$ and $\mathcal{S}$ a $\eta$-sparse
family of cubes. Then

\[
\sum_{Q\in\mathcal{S}}w(Q)\leq c_{n}[w]_{A_{\infty}}w\left(\bigcup_{Q\in\mathcal{S}}Q\right).
\]
\end{lem}

\begin{proof}
We can assume that $\mathcal{S}=\bigcup_{k=0}^{\infty}\mathcal{S}_{k}$
where $\{\mathcal{S}_{k}\}$ is an increasing sequence of finite sparse
families. Now we fix $k$ and consider $\mathcal{S}_{k}^{*}$ the
family of maximal cubes of $\mathcal{S}_{k}$ with respect to the
inclusion. Then 
\[
\sum_{Q\in\mathcal{S}_{k}}w(Q)=\sum_{Q\in\mathcal{S}_{k}^{*}}\sum_{P\in\mathcal{S}_{k},\,P\subset Q}w(P).
\]
Notice that 
\[
\begin{split}\sum_{P\in\mathcal{S}_{k},\,P\subset Q}w(P) & \leq\frac{1}{\eta}\sum_{P\in\mathcal{S}_{k},\,P\subset Q}\frac{1}{|P|}w(P)|E_{P}|\leq\frac{1}{\eta}\sum_{P\in\mathcal{S}_{k},\,P\subset Q}\frac{1}{|P|}w(P)|E_{P}|\\
 & \leq\frac{1}{\eta}\sum_{P\in\mathcal{S}_{k},\,P\subset Q}\inf_{z\in P}M(w\chi_{Q})(z)|E_{P}|\leq\frac{1}{\eta}\sum_{P\in\mathcal{S}_{k},\,P\subset Q}\int_{E_{P}}M(w\chi_{Q})\\
 & \leq\frac{1}{\eta}\int_{Q}M(w\chi_{Q})=\frac{1}{\eta}\frac{1}{w(Q)}\int_{Q}M(w\chi_{Q})w(Q)\leq\frac{1}{\eta}[w]_{A_{\infty}}w(Q).
\end{split}
\]
Hence
\[
\begin{split}\sum_{Q\in\mathcal{S}^{*}}\sum_{P\in\mathcal{S}_{k},\,P\subset Q}w(P) & \leq\frac{1}{\eta}[w]_{A_{\infty}}\sum_{Q\in\mathcal{S}_{k}^{*}}w(Q)=\frac{1}{\eta}[w]_{A_{\infty}}w\left(\bigcup_{Q\in\mathcal{S}_{k}^{*}}Q\right)\\
 & =\frac{1}{\eta}[w]_{A_{\infty}}w\left(\bigcup_{Q\in\mathcal{S}_{k}^{*}}Q\right)=\frac{1}{\eta}[w]_{A_{\infty}}w\left(\bigcup_{Q\in\mathcal{S}_{k}}Q\right)\\
 & \leq\frac{1}{\eta}[w]_{A_{\infty}}w\left(\bigcup_{Q\in\mathcal{S}}Q\right).
\end{split}
\]
Consequently
\[
\sum_{Q\in\mathcal{S}_{k}}w(Q)\leq\frac{1}{\eta}[w]_{A_{\infty}}w\left(\bigcup_{Q\in\mathcal{S}}Q\right)
\]
and letting $k\rightarrow\infty$ we are done.
\end{proof}
To end the section we provide some results related to singular weighted
maximal functions.
\begin{lem}
\label{Lem:MUnDyad}Let $A$ a Young function such that $A(st)\leq\kappa A(s)A(t)$.
Let $\mathcal{D}_{j}$ $j=1,\dots,k$ be dyadic grids and let $w$
a weight. Then
\[
w\left(\left\{ x\in\mathbb{R}^{n}:M_{A(w)}^{\mathcal{F}}f(x)>t\right\} \right)\leq\kappa c_{n}\int_{\mathbb{R}^{d}}A\left(\frac{|f(x)|}{t}\right)w(x)dx
\]
where $\mathcal{F}=\bigcup_{j=1}^{3^{n}}\mathcal{D}_{j}$.
\end{lem}

\begin{proof}
Let $t>0$. Notice that 
\[
M_{w}^{\mathcal{F}}f(x)\leq\sum_{j=1}^{3^{n}}M_{w}^{\mathcal{D}_{j}}f(x).
\]
Then taking into account that 
\[
w\left(\left\{ x\in\mathbb{R}^{n}:M_{w}^{\mathcal{D}}f(x)>\lambda\right\} \right)\leq\int_{\mathbb{R}^{d}}A\left(\frac{|f(x)|}{\lambda}\right)w(x)dx
\]
(see \cite[Section 15]{LN}) we have that
\[
\begin{split} & w\left(\left\{ x\in\mathbb{R}^{n}:M_{w}^{\mathcal{F}}f(x)>t\right\} \right)\\
 & \leq w\left(\left\{ x\in\mathbb{R}^{n}:\sum_{j=1}^{3^{n}}M_{w}^{\mathcal{D}_{j}}f(x)>t\right\} \right)\\
 & \leq\sum_{j=1}^{3^{n}}w\left(\left\{ x\in\mathbb{R}^{n}:M_{w}^{\mathcal{D}_{j}}f(x)>\frac{t}{3^{n}}\right\} \right)\\
 & \leq\sum_{j=1}^{3^{n}}\int_{\mathbb{R}^{d}}A\left(\frac{3^{n}|f(x)|}{t}\right)w(x)dx\\
 & \leq c_{n}\kappa\int_{\mathbb{R}^{d}}A\left(\frac{|f(x)|}{t}\right)w(x)dx
\end{split}
\]
and we are done.
\end{proof}

\subsection{Proof of Theorem \ref{Thm:1}}

\subsubsection{Calderón-Zygmund operators}

Using pointwise sparse domination it suffices to settle the result
for a sparse operator $A_{\mathcal{S}}$ where $\mathcal{S}$ is a
$\frac{8}{9}$-sparse family contained in a dyadic lattice $\mathcal{D}$.

Let $G=\{\frac{A_{\mathcal{S}}(fv)(x)}{v(x)}>1\}\setminus\{M_{uv}^{\mathcal{D}}(f)>\frac{1}{2}\}$
and assume that $\|f\|_{L^{1}(uv)}=1$. Then it suffices to prove
that 
\[
uv(G)\leq c_{n,p}[uv]_{A_{\infty}}[u]_{A_{1}}\log\left(e+[uv]_{A_{\infty}}[u]_{A_{1}}[v]_{A_{p}(u)}\right)+\frac{1}{2}uv(G).
\]
If we denote $g=\chi_{G}$ then 
\[
\begin{split}uv(G) & \lesssim c_{n}c_{T}\sum_{Q\in\mathcal{S}_{j}}\langle fv\rangle_{Q,1}\langle g\rangle_{Q,1}^{u}u(Q).\\
 & \leq c_{n}c_{T}[u]_{A_{1}}\sum_{Q\in\mathcal{S}}\langle f\rangle_{Q,1}^{uv}\langle g\rangle_{Q,1}^{u}uv(Q)
\end{split}
\]
and it suffices to prove that 
\[
\begin{split} & c_{n}c_{T}[u]_{A_{1}}\sum_{Q\in\mathcal{S}}\langle f\rangle_{Q,1}^{uv}\langle g\rangle_{Q,1}^{u}uv(Q)\\
 & \leq c_{n,p}[uv]_{A_{\infty}}[u]_{A_{1}}\log\left(e+[uv]_{A_{\infty}}[u]_{A_{1}}[v]_{A_{p}(u)}\right)+\frac{1}{2}uv(G).
\end{split}
\]
We split the sparse family as follows. Let $Q\in\mathcal{S}_{k,j}$,
$k,j\geq0$ if
\[
\begin{split}2^{-j-1} & <\langle f\rangle_{Q,1}^{uv}\leq2^{-j},\\
2^{-k-1} & <\langle g\rangle_{Q,1}^{u}\leq2^{-k}.
\end{split}
\]
Let us call
\[
s_{k,j}=\sum_{Q\in\mathcal{S}_{k,j}}\langle f\rangle_{Q,1}^{uv}\langle g\rangle_{Q,1}^{u}uv(Q).
\]
We claim that 
\[
s_{k,j}\leq\begin{cases}
c_{n}2^{-k}[uv]_{A_{\infty}},\\
c_{n,p}[uv]_{A_{\infty}}[v]_{A_{p}(u)}2^{-j}2^{k(p-1)}uv\left(G\right).
\end{cases}
\]
For the top estimate we argue as follows. Using Lemma \ref{Lem:DisEQ}
we have that there exist sets $\tilde{E}_{Q}\subset Q$ such that
\[
\sum_{Q\in\mathcal{S}_{k,j}}\chi_{\tilde{E_{Q}}}(x)\leq\left\lceil c_{n}[uv]_{A_{\infty}}\right\rceil 
\]
and
\[
\int_{Q}fuv\leq4\int_{\tilde{E_{Q}}}fuv.
\]
\[
\]
Then
\begin{equation}
\begin{split}s_{k,j} & \leq2^{-k}\sum_{Q\in\mathcal{S}_{j,k}}\int_{Q}fuv\\
 & \leq4\cdot2^{-k}\sum_{Q\in\mathcal{S}_{j,k}}\int_{\tilde{E_{Q}}}fuv\\
 & \leq c_{n}[uv]_{A_{\infty}}2^{-k}\int_{\mathbb{R}^{n}}fuv=c_{n}[uv]_{A_{\infty}}2^{-k}.
\end{split}
\label{eq:TopEst}
\end{equation}
For the lower estimate, using Lemma \ref{Lem:UnQ},
\[
\begin{split}s_{k,j} & \leq2^{-j}2^{-k}\sum_{Q\in\mathcal{S}_{j,k}}uv(Q)\\
 & \leq c_{n}[uv]_{A_{\infty}}2^{-j}2^{-k}uv\left(\bigcup_{Q\in\mathcal{S}_{j,k}}Q\right)\\
 & \leq c_{n}[uv]_{A_{\infty}}2^{-j}2^{-k}uv\left(\left\{ x\in\mathbb{R}^{n}\,:\,M_{u}g>2^{-k-1}\right\} \right).
\end{split}
\]
Now notice that since $v\in A_{p}(u)$, Lemma \ref{Lem:ContAvg} yields
\[
\frac{1}{u(Q)}\int_{Q}gu\leq\left(\frac{[v]_{A_{p}(u)}}{uv(Q)}\int_{Q}guv\right)^{\frac{1}{p}}.
\]
Taking that into account, by Lemma \ref{Lem:MUnDyad}
\[
\begin{split} & \leq c_{n}[uv]_{A_{\infty}}2^{-j}2^{-k}uv\left(\left\{ x\in\mathbb{R}^{n}\,:\,M_{u}g>2^{-k-1}\right\} \right)\\
 & \leq c_{n}[uv]_{A_{\infty}}2^{-j}2^{-k}uv\left(\left\{ x\in\mathbb{R}^{n}\,:\,M_{uv}g>2^{-kp-p}[v]_{A_{p}(u)}^{-1}\right\} \right)\\
 & \leq c_{n,p}[uv]_{A_{\infty}}[v]_{A_{p}(u)}2^{-j}2^{k(p-1)}uv\left(G\right).
\end{split}
\]
Combining the estimates above
\[
\begin{split}uv(G) & \leq c_{n}c_{T}[u]_{A_{1}}\sum_{k=0}^{\infty}\sum_{j=0}^{\infty}s_{k,j}\\
 & \leq\sum_{k=0}^{\infty}\sum_{j=0}^{\infty}\min\left\{ c_{n}c_{T}2^{-k}[u]_{A_{1}}[uv]_{A_{\infty}},c_{n,p}[u]_{A_{1}}[uv]_{A_{\infty}}[v]_{A_{p}(u)}2^{-j}2^{-k}2^{kp}uv\left(G\right)\right\} 
\end{split}
.
\]
Now we are left with estimating the double sum. Applying Lemma \ref{Lem:DoubleSum}
with 
\[
\begin{split}\gamma_{1} & =c_{n}c_{T}[u]_{A_{1}}[uv]_{A_{\infty}}\\
\gamma_{2} & =c_{n,p}[u]_{A_{1}}[uv]_{A_{\infty}}[v]_{A_{p}(u)}
\end{split}
\]
$\delta=p$, $\beta=uv(G)$, $\rho_{1}=\rho_{2}=0$ and $\gamma=1$
we are done.

\subsubsection{Commutators}

Using pointwise sparse domination it suffices to settle the result
for suitable dyadic operators. Let 
\[
G=\left\{ \frac{\sum_{Q\in\mathcal{S}}|b-b_{Q}|^{m-h}\chi_{Q}\frac{1}{|Q|}\int_{Q}|b-b_{Q}|^{h}fv}{v(x)}>1\right\} \setminus\left\{ M_{L(\log L)^{\frac{h}{r}}(uv)}^{\mathcal{D}}(f)>\frac{1}{2}\right\} .
\]
Assume that $\|b\|_{Osc_{\exp L^{r}}}=1$. It suffices to prove that
\[
uv(G)\leq c\varphi_{m,h}(u,v)\int_{\mathbb{R}^{n}}\Phi_{\frac{h}{r}}(|f(x)|)dx+\frac{1}{2}uv(G)
\]
where
\[
\varphi_{m,h}(u,v)=[u]_{A_{1}}[uv]_{A_{\infty}}^{1+\frac{h}{r}}[u]_{A_{\infty}}^{\frac{m-h}{r}}\log\left(e+[u]_{A_{1}}[uv]_{A_{\infty}}^{1+\frac{h}{r}}[u]_{A_{\infty}}^{\frac{m-h}{r}}[v]_{A_{p}(u)}\right)^{1+\frac{h}{r}}.
\]
If we denote $g=\chi_{G}$ then 
\[
\begin{split}uv(G) & \leq\sum_{Q\in\mathcal{S}}\frac{1}{|Q|}\int_{Q}|b-b_{Q}|^{h}fv\int_{Q}|b-b_{Q}|^{m-h}u\\
 & \leq\sum_{Q\in\mathcal{S}}\frac{u(Q)}{|Q|}\int_{Q}|b-b_{Q}|^{h}fv\frac{1}{u(Q)}\int_{Q}|b-b_{Q}|^{m-h}u\\
 & \leq[u]_{A_{1}}\sum_{Q\in\mathcal{S}}\frac{1}{uv(Q)}\int_{Q}|b-b_{Q}|^{h}fuv\frac{1}{u(Q)}\int_{Q}|b-b_{Q}|^{m-h}uuv(Q)\\
 & \leq[u]_{A_{1}}\sum_{Q\in\mathcal{S}}\left(\||b-b_{Q}|^{h}\|_{\exp L^{r/h}(uv),Q}\|f\|_{L(\log L)^{\frac{h}{r}}(uv),Q}\right.\\
 & \quad\left.\times\||b-b_{Q}|^{m-h}\|_{\exp L^{r/(m-h)}(uv),Q}\|g\|_{L(\log L)^{\frac{m-h}{r}}(u),Q}uv(Q)\right)\\
 & \leq c\|b\|_{Osc_{\exp L^{r}}}^{m}[uv]_{A_{\infty}}^{\frac{h}{r}}[u]_{A_{\infty}}^{\frac{m-h}{r}}\sum_{Q\in\mathcal{S}}\|f\|_{L(\log L)^{\frac{h}{r}}(uv),Q}\|g\|_{L(\log L)^{\frac{m-h}{r}}(u),Q}uv(Q)\\
 & =c[uv]_{A_{\infty}}^{\frac{h}{r}}[u]_{A_{\infty}}^{\frac{m-h}{r}}\sum_{Q\in\mathcal{S}}\|f\|_{L(\log L)^{\frac{h}{r}}(uv),Q}\|g\|_{L(\log L)^{\frac{m-h}{r}}(u),Q}uv(Q).
\end{split}
\]
We split the sparse family as follows $Q\in\mathcal{S}_{k,j}$, $k,j\geq0$
if
\[
\begin{split}2^{-j-1} & <\|f\|_{L(\log L)^{\frac{h}{r}}(uv),Q}\leq2^{-j},\\
2^{-k-1} & <\|g\|_{L(\log L)^{\frac{m-h}{r}}(u),Q}\leq2^{-k}.
\end{split}
\]
Then
\[
\sum_{Q\in\mathcal{S}}\|f\|_{L(\log L)^{\frac{h}{r}}(uv),Q}\|g\|_{L(\log L)^{\frac{m-h}{r}}(u),Q}uv(Q)=\sum_{k,j\geq0}s_{k,j}.
\]
Now we observe that 
\[
s_{k,j}\leq\begin{cases}
c_{n}[uv]_{A_{\infty}}2^{-k}j^{\frac{h}{r}},\\
c_{n,p,m}[uv]_{A_{\infty}}[v]_{A_{p}(u)}2^{-j}2^{k(p-1)}k^{\frac{m-h}{r}p}uv(G)uv\left(G\right).
\end{cases}
\]
For the top estimate we use Lemma \ref{Lem:DisEQ} with $w=uv$ and
$A(t)=\Phi_{\frac{h}{r}}(t)$, and we have that
\[
uv(Q)\|f\|_{L(\log L)^{\frac{h}{r}}(uv),Q}\leq cj^{\frac{h}{r}}\int_{\tilde{E_{Q}}}\Phi_{\frac{h}{r}}\left(|f|\right)uv.
\]
with
\[
\sum_{Q\in\mathcal{S}_{k,j}}\chi_{\tilde{E_{Q}}}(x)\leq\left\lceil c_{n}[uv]_{A_{\infty}}\right\rceil .
\]
Then
\[
\begin{split}s_{k,j} & \leq2^{-k}j^{\frac{h}{r}}\sum_{Q\in\mathcal{S}_{j,k}}\int_{\tilde{E_{Q}}}\Phi_{\frac{h}{r}}\left(|f|\right)uv.\\
 & \leq2\cdot2^{-k}j^{\frac{h}{r}}\sum_{Q\in\mathcal{S}_{j,k}}\int_{\tilde{E_{Q}}}\Phi_{\frac{h}{r}}\left(|f|\right)uv.\\
 & \leq c_{n}[uv]_{A_{\infty}}2^{-k}j^{\frac{h}{r}}\int_{\mathbb{R}^{n}}\Phi_{\frac{h}{r}}\left(|f|\right)uv.
\end{split}
\]
For the lower estimate, by Lemma \ref{Lem:UnQ}
\[
\begin{split}s_{k,j} & \leq2^{-j}2^{-k}\sum_{Q\in\mathcal{S}_{j,k}}uv(Q)\\
 & =c[uv]_{A_{\infty}}2^{-j}2^{-k}uv\left(\bigcup_{Q\in\mathcal{S}_{j,k}}Q\right)\\
 & \leq c_{n}[uv]_{A_{\infty}}2^{-j}2^{-k}uv\left(\left\{ x\in\mathbb{R}^{n}\,:\,M_{L(\log L)^{\frac{m-h}{r}}(u)}g>2^{-k-1}\right\} \right).
\end{split}
\]
Taking into account Lemma \ref{Lem:ContAvg}
\[
\|g\|_{L(\log L)^{\frac{m-h}{r}}(u)}\leq\|g\|_{[v]_{A_{p}(u)}L^{p}(\log L)^{p\frac{m-h}{r}}(uv).}
\]
That estimate combined with Lemma \ref{Lem:MUnDyad} allows us to
argue as follows
\[
\begin{split} & \leq c_{n}[uv]_{A_{\infty}}2^{-j}2^{-k}uv\left(\left\{ x\in\mathbb{R}^{n}\,:\,M_{[v]_{A_{p}(u)}L^{p}(\log L)^{p\frac{m-h}{r}}(uv)}g>2^{-k-1}\right\} \right)\\
 & \leq c_{n,p}[uv]_{A_{\infty}}2^{-j}2^{-k}\int_{\mathbb{R}^{d}}[v]_{A_{p}(u)}\Phi_{\frac{m-h}{r}}\left(2^{k+1}g\right)^{p}.\\
 & \leq c_{n,p}[uv]_{A_{\infty}}[v]_{A_{p}(u)}2^{-j}2^{-k}\Phi_{\frac{m-h}{r}}\left(2^{k+1}\right)^{p}uv(G).\\
 & \leq c_{n,p}[uv]_{A_{\infty}}[v]_{A_{p}(u)}2^{-j}2^{-k}2^{kp+p}\log\left(e+2^{k+1}\right)^{\frac{m-h}{r}p}uv(G).\\
 & \leq c_{n,p,m}[uv]_{A_{\infty}}[v]_{A_{p}(u)}2^{-j}2^{k(p-1)}k^{\frac{m-h}{r}p}uv(G).
\end{split}
\]
Combining the estimates above
\[
\begin{split}uv(G) & \leq c[uv]_{A_{\infty}}^{\frac{h}{r}}[u]_{A_{\infty}}^{\frac{m-h}{r}}[u]_{A_{1}}\sum_{k=0}^{\infty}\sum_{j=0}^{\infty}s_{k,j}\\
 & \leq\sum_{k=0}^{\infty}\sum_{j=0}^{\infty}\min\left\{ cc_{n,p,m}[uv]_{A_{\infty}}^{1+\frac{h}{r}}[u]_{A_{\infty}}^{\frac{m-h}{r}}[u]_{A_{1}}[v]_{A_{p}(u)}2^{-j}2^{k(p-1)}k^{\frac{m-h}{r}p}uv(G),\right.\\
 & \left.cc_{n}[uv]_{A_{\infty}}^{1+\frac{h}{r}}[u]_{A_{\infty}}^{\frac{m-h}{r}}[u]_{A_{1}}2^{-k}j^{\frac{h}{r}}\int_{\mathbb{R}^{n}}\Phi_{\frac{h}{r}}\left(|f|\right)uv\right\} .
\end{split}
\]
We end up the proof applying Lemma \ref{Lem:DoubleSum}, with $\gamma_{1}=cc_{n}[u]_{A_{1}}[uv]_{A_{\infty}}^{1+\frac{h}{r}}[u]_{A_{\infty}}^{\frac{m-h}{r}}\int_{\mathbb{R}^{n}}\Phi_{\frac{h}{r}}\left(|f|\right)uv$,
$\gamma_{2}=cc_{n,p,m}[u]_{A_{1}}[uv]_{A_{\infty}}^{1+\frac{h}{r}}[u]_{A_{\infty}}^{\frac{m-h}{r}}[v]_{A_{p}(u)}$,
$\beta=uv\left(G\right)$, $\delta=p-1$, $\gamma=1$, $\rho_{1}=\frac{h}{r}$
and $\rho_{2}=\frac{m-h}{r}p$.

\subsubsection{Rough singular integrals}

Let us fix a dyadic lattice $\mathcal{D}$ and let $\mathcal{D}_{j}$
$j=1,\dots,3^{n}$ obtained using the $3^{n}$ dyadic lattices trick
(Lemma \ref{Lem:3ndlt}). Now let
\[
G=\left\{ \frac{T_{\Omega}(fv)(x)}{v(x)}>1\right\} \setminus\left\{ M_{uv}^{\mathcal{F}}(f)>\frac{1}{2}\right\} 
\]
where $\mathcal{F=\bigcup}_{j=1}^{3^{n}}\mathcal{D}_{j}$ and assume
that $\|f\|_{L^{1}(uv)}=1$. 

Then it suffices to prove that 
\[
uv(G)\leq c_{n,p}[uv]_{A_{\infty}}[u]_{A_{\infty}}[u]_{A_{1}}\log\left(e+[uv]_{A_{\infty}}[u]_{A_{\infty}}[u]_{A_{1}}[v]_{A_{p}(u)}\right)+\frac{1}{2}uv(G).
\]
Note that 
\[
uv(G)\leq\left|\int_{\mathbb{R}^{n}}\frac{T_{\Omega}(fv)}{v}uvg\right|=\left|\int_{\mathbb{R}^{n}}T_{\Omega}(fv)ug\right|
\]
where $g\simeq\chi_{G}$. Then for $s=1+\frac{1}{2\tau_{n}[u]_{A_{\infty}}}$,
notice that, arguing as in \cite{LiPRR}
\[
\frac{u^{s}(G\cap Q)}{u^{s}(Q)}\lesssim\left(\frac{u(G\cap Q)}{u(Q)}\right)^{\frac{1}{2}}.
\]
Taking into account \eqref{eq:Rough} and Remark \ref{Rem:3nRough}
we have that
\[
\begin{split}uv(G) & \lesssim c_{n}c_{T}s'\sum_{j=1}^{3^{n}}\sum_{Q\in\mathcal{S}_{j}}\langle fv\rangle_{Q,1}\langle\chi_{G}u\rangle_{Q,s}\\
 & =c_{n}c_{T}s'\sum_{j=1}^{3^{n}}\sum_{Q\in\mathcal{S}_{j}}\langle fv\rangle_{Q,1}\langle\chi_{G}\rangle_{Q,s}^{u^{s}}\langle u\rangle_{Q,s}\\
 & \leq c_{n}c_{T}[u]_{A_{\infty}}[u]_{A_{1}}\sum_{j=1}^{3^{n}}\sum_{Q\in\mathcal{S}_{j}}\langle f\rangle_{Q,1}^{uv}\langle g\rangle_{Q,2s}^{u}uv(Q)
\end{split}
\]
where $g=\chi_{G}$ and each $\mathcal{S}_{j}\subset\mathcal{D}_{j}$. 

At this point one remark is in order. Notice that in this case, since
we don't have pointwise domination, we need to remove the cubes where
the maximal function is large from the sparse family using just one
maximal function. On the other hand if we choose the standard maximal
function instead of some dyadic version that would lead to some dependence
on the doubling constant of the measure $uvdx$, which is something
that we avoid with our choice (see Lemma \ref{Lem:MUnDyad}). 

After that remark we continue with the proof. Notice that it suffices
to prove that for each $j$, 
\[
\begin{split} & c_{n}c_{T}[u]_{A_{\infty}}[u]_{A_{1}}\sum_{Q\in\mathcal{S}_{j}}\langle f\rangle_{Q,1}^{uv}\langle g\rangle_{Q,2s}^{u}uv(Q)\\
 & \leq c_{n,p}[uv]_{A_{\infty}}[u]_{A_{\infty}}[u]_{A_{1}}\log\left(e+[uv]_{A_{\infty}}[u]_{A_{\infty}}[u]_{A_{1}}[v]_{A_{p}(u)}\right)+\frac{1}{2\cdot3^{n}}uv(G).
\end{split}
\]
Taking into account the definition of $G$, since we remove the set
where $M_{uv}^{\mathcal{F}}(f)>\frac{1}{2}$ we can split sparse family
as follows $Q\in\mathcal{S}_{k,j}$, $k,j\geq0$ if
\[
\begin{split}2^{-j-1} & <\langle f\rangle_{Q,1}^{uv}\leq2^{-j}\\
2^{-k-1} & <\langle g\rangle_{Q,2s}^{u}\leq2^{-k}.
\end{split}
\]
Let us call
\[
s_{k,j}=\sum_{Q\in\mathcal{S}_{k,j}}\langle f\rangle_{Q,1}^{uv}\langle g\rangle_{Q,2s}^{u}uv(Q).
\]
Now we observe that 
\[
s_{k,j}\leq\begin{cases}
c_{n}2^{-k}[uv]_{A_{\infty}}\\
c_{n,p}[uv]_{A_{\infty}}[v]_{A_{p}(u)}2^{-j}2^{(2ps-1)k}uv\left(G\right).
\end{cases}
\]
For the top estimate we argue as we did in \eqref{eq:TopEst}.  For
the lower estimate, using Lemma \ref{Lem:UnQ}, 
\[
\begin{split}s_{k,j} & \leq2^{-j}2^{-k}\sum_{Q\in\mathcal{S}_{j,k}}uv(Q)\\
 & \leq c_{n}[uv]_{A_{\infty}}2^{-j}2^{-k}uv\left(\bigcup_{Q\in\mathcal{S}_{j,k}}Q\right)\\
 & \leq c_{n}[uv]_{A_{\infty}}2^{-j}2^{-k}uv\left(\left\{ x\in\mathbb{R}^{n}\,:\,(M_{u}g)^{\frac{1}{2s}}>2^{-k-1}\right\} \right).
\end{split}
\]
Since $v\in A_{p}(u)$, taking into account Lemmas \ref{Lem:ContAvg}
and \ref{Lem:MUnDyad}
\[
\begin{split} & \leq c_{n}[uv]_{A_{\infty}}2^{-j}2^{-k}uv\left(\left\{ x\in\mathbb{R}^{n}\,:\,\left(M_{u}g\right)^{\frac{1}{2s}}>2^{-k-1}\right\} \right)\\
 & \leq c_{n}[uv]_{A_{\infty}}2^{-j}2^{-k}uv\left(\left\{ x\in\mathbb{R}^{n}\,:\,\left([v]_{A_{p}(u)}M_{uv}g\right)^{\frac{1}{2sp}}>2^{-k-1}\right\} \right)\\
 & \leq c_{n}[uv]_{A_{\infty}}2^{-j}2^{-k}uv\left(\left\{ x\in\mathbb{R}^{n}\,:\,M_{uv}g>2^{-2spk-2sp}[v]_{A_{p}(u)}^{-1}\right\} \right)\\
 & \leq c_{n,p}[uv]_{A_{\infty}}[v]_{A_{p}(u)}2^{-j}2^{(2ps-1)k}uv\left(G\right).
\end{split}
\]
Combining the estimates above
\[
\begin{split}uv(G) & \leq c_{n}c_{T}\sum_{k=0}^{\infty}\sum_{j=0}^{\infty}s_{k,j}\\
 & \leq\sum_{k=0}^{\infty}\sum_{j=0}^{\infty}\min\left\{ c_{n}c_{T}\beta_{u,v}2^{-k},c_{n,p}\beta_{u,v}[v]_{A_{p}(u)}2^{-j}2^{(2ps-1)k}uv\left(G\right)\right\} 
\end{split}
\]
where $\beta_{u,v}=[uv]_{A_{\infty}}[u]_{A_{\infty}}[u]_{A_{1}}$.
We end the proof using Lemma \ref{Lem:DoubleSum}, with $\gamma_{1}=c_{n}c_{T}\beta_{u,v}$,
$\gamma_{2}=c_{n,p}\beta_{u,v}[v]_{A_{p}(u)}$, $\beta=uv\left(G\right)$,
$\delta=2ps-1$, $\gamma=3^{n}$ and $\rho_{1}=\rho_{2}=0$.

\subsection{Proof of Theorem \ref{Thm:2}}

\subsubsection{Calderón-Zygmund operators}

Using pointwise sparse domination it suffices to settle the result
for a sparse operator $A_{\mathcal{S}}$ where $\mathcal{S}$ is a
$\frac{8}{9}$-sparse family.

Let $G=\{\frac{A_{\mathcal{S}}(fv)(x)}{v(x)}>1\}\setminus\{M_{v}^{\mathcal{D}}(f)>\frac{1}{2}\}$
and assume that $f\geq0$ and $\|f\|_{L^{1}(uv)}=1$. Then it suffices
to prove that 
\[
uv(G)\leq c_{n,p}[v]_{A_{1}}[v]_{A_{\infty}}[u]_{A_{1}(v)}\log\left(e+[uv]_{A_{\infty}}[v]_{A_{1}}[u]_{A_{1}(v)}\right)+\frac{1}{2}uv(G).
\]
If we denote $g=\chi_{G}$ then 
\[
\]
\[
\begin{split}uv(G) & \lesssim c_{n}c_{T}\sum_{Q\in\mathcal{S}}\langle fv\rangle_{Q,1}\int_{Q}gu\\
 & =c_{n}c_{T}\sum_{Q\in\mathcal{S}_{j}}\langle f\rangle_{Q,1}^{v}\frac{v(Q)}{|Q|}\int_{Q}gu\\
 & \leq[v]_{A_{1}}c_{n}c_{T}\sum_{Q\in\mathcal{S}_{j}}\langle f\rangle_{Q,1}^{v}\int_{Q}guv.\\
 & =c_{n}c_{T}[v]_{A_{1}}\sum_{Q\in\mathcal{S}_{j}}\langle f\rangle_{Q,1}^{v}\langle g\rangle_{Q,1}^{uv}uv(Q)
\end{split}
\]
and it suffices to prove that
\[
\begin{split} & c_{n}c_{T}[v]_{A_{1}}\sum_{Q\in\mathcal{S}}\langle f\rangle_{Q,1}^{v}\langle g\rangle_{Q,1}^{uv}uv(Q)\\
 & \leq c_{n,p}[v]_{A_{1}}[v]_{A_{\infty}}[u]_{A_{1}(v)}\log\left(e+[uv]_{A_{\infty}}[v]_{A_{1}}[u]_{A_{1}(v)}\right)+\frac{1}{2}uv(G).
\end{split}
\]
We split the sparse family as follows. Let $Q\in\mathcal{S}_{k,j}$,
$k,j\geq0$ if
\[
\begin{split}2^{-j-1} & <\langle f\rangle_{Q,1}^{v}\leq2^{-j}\\
2^{-k-1} & <\langle g\rangle_{Q,1}^{uv}\leq2^{-k}
\end{split}
\]
Let us call 
\[
s_{k,j}=\sum_{Q\in\mathcal{S}_{k,j}}\langle f\rangle_{Q,1}^{v}\langle g\rangle_{Q,1}^{uv}uv(Q).
\]
Now we observe that 
\[
s_{k,j}\leq\begin{cases}
c_{n}2^{-k}[u]_{A_{1}(v)}[v]_{A_{\infty}}\\
c_{n}[uv]_{A_{\infty}}2^{-j}2^{-k}uv(G).
\end{cases}
\]
For the top estimate we argue as follows. Using Lemma \ref{Lem:DisEQ}
we have that

\[
\int_{Q}fv\leq4\int_{\tilde{E_{Q}}}fv
\]
where $\tilde{E}_{Q}\subset Q$ and
\[
\sum_{Q\in\mathcal{S}_{k,j}}\chi_{\tilde{E_{Q}}}(x)\leq\left\lceil c_{n}[v]_{A_{\infty}}\right\rceil .
\]
Then
\begin{equation}
\begin{split}s_{k,j} & \leq2^{-k}\sum_{Q\in\mathcal{S}_{j,k}}\frac{uv(Q)}{v(Q)}\int_{Q}fv\\
 & \leq2\cdot2^{-k}\sum_{Q\in\mathcal{S}_{j,k}}\frac{uv(Q)}{v(Q)}\int_{\tilde{E_{Q}}}fv\\
 & \leq2\cdot2^{-k}[u]_{A_{1}(v)}\sum_{Q\in\mathcal{S}_{j,k}}\int_{\tilde{E_{Q}}}fuv\\
 & \leq c_{n}2^{-k}[u]_{A_{1}(v)}[v]_{A_{\infty}}\int_{\mathbb{R}^{n}}fuv.\\
 & =c_{n}2^{-k}[u]_{A_{1}(v)}[v]_{A_{\infty}}.
\end{split}
\label{eq:TopThm2}
\end{equation}
For the lower estimate, using Lemma \ref{Lem:UnQ},
\[
\begin{split}s_{k,j} & \leq2^{-j}2^{-k}\sum_{Q\in\mathcal{S}_{j,k}}uv(Q)\\
 & \leq c_{n}[uv]_{A_{\infty}}2^{-j}2^{-k}uv\left(\bigcup_{Q\in\mathcal{S}_{j,k}}Q\right)\\
 & \leq c_{n}[uv]_{A_{\infty}}2^{-j}2^{-k}uv\left(\left\{ x\in\mathbb{R}^{n}\,:\,M_{uv}^{\mathcal{D}}(g)>2^{-k-1}\right\} \right).
\end{split}
\]
Now using the weak-type $(1,1)$ of $M_{uv}$ (Lemma \ref{Lem:MUnDyad})
\[
\begin{split} & \leq c_{n}[uv]_{A_{\infty}}2^{-j}2^{-k}uv\left(\left\{ x\in\mathbb{R}^{n}\,:\,M_{uv}^{\mathcal{D}}(g)>2^{-k-1}\right\} \right)\\
 & \leq c_{n}[uv]_{A_{\infty}}2^{-j}2^{-k}uv(G).
\end{split}
\]
Combining the estimates above,
\[
\begin{split}uv(G) & \leq c_{n}c_{T}[v]_{A_{1}}\sum_{k=0}^{\infty}\sum_{j=0}^{\infty}s_{k,j}\\
 & \leq\sum_{k=0}^{\infty}\sum_{j=0}^{\infty}\min\left\{ c_{n}c_{T}2^{-k}[v]_{A_{1}}[v]_{A_{\infty}}[u]_{A_{1}(v)},c_{n}[v]_{A_{1}}[uv]_{A_{\infty}}2^{-j}2^{-k}uv\left(G\right)\right\} .
\end{split}
\]
An application of Lemma \ref{Lem:DoubleSum} with 
\[
\gamma_{1}=c_{n}c_{T}[v]_{A_{1}}[v]_{A_{\infty}}[u]_{A_{1}(v)}\qquad\gamma_{2}=c_{n}[v]_{A_{1}}[uv]_{A_{\infty}}
\]
$\delta=0$, $\beta=uv(G)$, $\gamma=1$ and $\rho_{1}=\rho_{2}=0$
ends the proof.

\subsubsection{Commutators}

Using pointwise sparse domination it suffices to settle the result
for suitable dyadic operators. Let 
\[
G=\left\{ \frac{\sum_{Q\in\mathcal{S}}|b(x)-b_{Q}|^{m-h}\chi_{Q}(x)\frac{1}{|Q|}\int_{Q}|b-b_{Q}|^{h}fv}{v(x)}>1\right\} \setminus\left\{ M_{L(\log L)^{\frac{h}{r}}(v)}^{\mathcal{D}}(f)>\frac{1}{2}\right\} 
\]
Assume that $\|b\|_{Osc_{\exp L^{r}}}=1$. It suffices to prove that
\[
uv(G)\leq c\varphi_{m,h}(u,v)\int_{\mathbb{R}^{n}}\Phi_{\frac{h}{r}}(|f(x)|)dx+\frac{1}{2}uv(G).
\]
where$\gamma_{1}=cc_{n}[u]_{A_{1}}[uv]_{A_{\infty}}^{1+\frac{h}{r}}[u]_{A_{\infty}}^{\frac{m-h}{r}}\int_{\mathbb{R}^{n}}\Phi_{\frac{h}{r}}\left(|f|\right)uv$,
$\gamma_{2}=cc_{n,p,m}[u]_{A_{1}}[uv]_{A_{\infty}}^{1+\frac{h}{r}}[u]_{A_{\infty}}^{\frac{m-h}{r}}[v]_{A_{p}(u)}$,
\[
\varphi_{m,h}(u,v)=[u]_{A_{1}}[uv]_{A_{\infty}}^{1+\frac{h}{r}}[u]_{A_{\infty}}^{\frac{m-h}{r}}\log\left(e+[u]_{A_{1}}[uv]_{A_{\infty}}^{1+\frac{h}{r}}[u]_{A_{\infty}}^{\frac{m-h}{r}}[v]_{A_{p}(u)}\right)^{1+\frac{h}{r}}
\]
If we denote $g=\chi_{G}$ then 
\[
\begin{split}uv(G) & \leq\sum_{Q\in\mathcal{S}}\frac{1}{|Q|}\int_{Q}|b-b_{Q}|^{h}fv\int_{Q}|b-b_{Q}|^{m-h}u\\
 & \leq\sum_{Q\in\mathcal{S}}\frac{1}{v(Q)}\int_{Q}|b-b_{Q}|^{h}fv\frac{v(Q)}{|Q|}\int_{Q}|b-b_{Q}|^{m-h}u\\
 & \leq[v]_{A_{1}}\sum_{Q\in\mathcal{S}}\frac{1}{v(Q)}\int_{Q}|b-b_{Q}|^{h}fv\frac{1}{uv(Q)}\int_{Q}|b-b_{Q}|^{m-h}uvuv(Q)\\
 & \leq[v]_{A_{1}}\sum_{Q\in\mathcal{S}}\left(\||b-b_{Q}|^{h}\|_{\exp L^{r/h}(v),Q}\|f\|_{L(\log L)^{\frac{h}{r}}(v),Q}\right.\\
 & \quad\times\left.\||b-b_{Q}|^{m-h}\|_{\exp L^{r/(m-h)}(uv),Q}\|g\|_{L(\log L)^{\frac{m-h}{r}}(uv),Q}uv(Q)\right)\\
 & \leq c\|b\|_{Osc_{\exp L^{r}}}^{m}[v]_{A_{1}}[v]_{A_{\infty}}^{\frac{h}{r}}[uv]_{A_{\infty}}^{\frac{m-h}{r}}\sum_{Q\in\mathcal{S}}\|f\|_{L(\log L)^{\frac{h}{r}}(v),Q}\|g\|_{L(\log L)^{\frac{m-h}{r}}(uv),Q}uv(Q)\\
 & =c[v]_{A_{1}}[v]_{A_{\infty}}^{\frac{h}{r}}[uv]_{A_{\infty}}^{\frac{m-h}{r}}\sum_{Q\in\mathcal{S}}\|f\|_{L(\log L)^{\frac{h}{r}}(v),Q}\|g\|_{L(\log L)^{\frac{m-h}{r}}(uv),Q}uv(Q)
\end{split}
\]
Let us split the sparse family as follows. Let $Q\in\mathcal{S}_{k,j}$,
$k,j\geq0$ if
\[
\begin{split}2^{-j-1} & <\|f\|_{L(\log L)^{\frac{h}{r}}(v),Q}\leq2^{-j}\\
2^{-k-1} & <\|g\|_{L(\log L)^{\frac{m-h}{r}}(uv),Q}\leq2^{-k}.
\end{split}
\]
Let us call
\[
s_{k,j}=\sum_{Q\in\mathcal{S}_{k,}j}\|f\|_{L(\log L)^{\frac{h}{r}}(v),Q}\|g\|_{L(\log L)^{\frac{m-h}{r}}(uv),Q}uv(Q).
\]
Now we observe that 
\[
s_{k,j}\leq\begin{cases}
c_{n}[u]_{A_{1}(v)}[v]_{A_{\infty}}2^{-k}j^{\frac{h}{r}}\int_{\mathbb{R}^{n}}\Phi_{\frac{h}{r}}\left(|f|\right)uv\\
c_{n,p,m}c[uv]_{A_{\infty}}2^{-j}2^{k(p-1)}k^{\frac{m-h}{r}p}uv(G).
\end{cases}
\]
For the top estimate we use Lemma \ref{Lem:DisEQ} with $w=v$ and
$A(t)=\Phi_{\frac{h}{r}}\left(t\right)$, and we have that
\[
v(Q)\|f\|_{L(\log L)^{\frac{h}{r}}(v),Q}\leq cj^{\frac{h}{r}}\int_{\tilde{E_{Q}}}\Phi_{\frac{h}{r}}\left(|f|\right)uv.
\]
with
\[
\sum_{Q\in\mathcal{S}_{k,j}}\chi_{\tilde{E_{Q}}}(x)\leq\left\lceil c_{n}[v]_{A_{\infty}}\right\rceil .
\]
Then
\[
\begin{split}s_{k,j} & \leq2^{-k}j^{\frac{h}{r}}\sum_{Q\in\mathcal{S}_{j,k}}\frac{uv(Q)}{v(Q)}\int_{\tilde{E_{Q}}}\Phi_{\frac{h}{r}}\left(|f|\right)v.\\
 & \leq2[u]_{A_{1}(v)}2^{-k}j^{\frac{h}{r}}\sum_{Q\in\mathcal{S}_{j,k}}\int_{\tilde{E_{Q}}}\Phi_{\frac{h}{r}}\left(|f|\right)uv.\\
 & \leq c_{n}[u]_{A_{1}(v)}[v]_{A_{\infty}}2^{-k}j^{\frac{h}{r}}\int_{\mathbb{R}^{n}}\Phi_{\frac{h}{r}}\left(|f|\right)uv.
\end{split}
\]
For the lower estimate, by Lemma \ref{Lem:MUnDyad}
\[
\begin{split}s_{k,j} & \leq2^{-j}2^{-k}\sum_{Q\in\mathcal{S}_{j,k}}uv(Q)\\
 & =c[uv]_{A_{\infty}}2^{-j}2^{-k}uv\left(\bigcup_{Q\in\mathcal{S}_{j,k}}Q\right)\\
 & \leq c_{n}[uv]_{A_{\infty}}2^{-j}2^{-k}uv\left(\left\{ x\in\mathbb{R}^{n}\,:\,M_{L(\log L)^{\frac{m-h}{r}}(uv)}^{\mathcal{D}}g>2^{-k-1}\right\} \right)\\
 & \leq c_{n,p}[uv]_{A_{\infty}}2^{-j}2^{-k}\Phi_{\frac{m-h}{r}}\left(2^{k+1}\right)^{p}uv(G).\\
 & \leq c_{n,p}[uv]_{A_{\infty}}2^{-j}2^{-k}2^{kp+p}\log\left(e+2^{k+1}\right)^{\frac{m-h}{r}p}uv(G).\\
 & \leq c_{n,p,m}[uv]_{A_{\infty}}2^{-j}2^{k(p-1)}k^{\frac{m-h}{r}p}uv(G).
\end{split}
\]
Combining the estimates above

\[
\begin{split}uv(G) & \leq c[v]_{A_{1}}[v]_{A_{\infty}}^{\frac{h}{r}}[uv]_{A_{\infty}}^{\frac{m-h}{r}}\sum_{k=0}^{\infty}\sum_{j=0}^{\infty}s_{k,j}\\
 & \leq\sum_{k=0}^{\infty}\sum_{j=0}^{\infty}\min\left\{ cc_{n,p,m}[v]_{A_{1}}[v]_{A_{\infty}}^{\frac{h}{r}}[uv]_{A_{\infty}}^{\frac{m-h}{r}}2^{-j}2^{k(p-1)}k^{\frac{m-h}{r}p}uv(G),\right.\\
 & \qquad\qquad\quad\left.cc_{n}[v]_{A_{1}}[v]_{A_{\infty}}^{\frac{h}{r}}[uv]_{A_{\infty}}^{\frac{m-h}{r}}[u]_{A_{1}(v)}[v]_{A_{\infty}}2^{-k}j^{\frac{h}{r}}\int_{\mathbb{R}^{n}}\Phi_{\frac{h}{r}}\left(|f|\right)uv\right\} 
\end{split}
\]
We end up the proof applying Lemma \ref{Lem:DoubleSum}, with 
\[
\begin{split}\gamma_{1} & =cc_{n}[v]_{A_{1}}[v]_{A_{\infty}}^{\frac{h}{r}}[uv]_{A_{\infty}}^{\frac{m-h}{r}}[u]_{A_{1}(v)}[v]_{A_{\infty}}2^{-k}j^{\frac{h}{r}}\int_{\mathbb{R}^{n}}\Phi_{\frac{h}{r}}\left(|f|\right)uv,\\
\gamma_{2} & =cc_{n,p,m}[v]_{A_{1}}[v]_{A_{\infty}}^{\frac{h}{r}}[uv]_{A_{\infty}}^{\frac{m-h}{r}}2^{-j}2^{k(p-1)}k^{\frac{m-h}{r}p},
\end{split}
\]
$\beta=uv\left(G\right)$, $\delta=p-1$, $\gamma=1$, $\rho_{1}=\frac{h}{r}$
and $\rho_{2}=\frac{m-h}{r}p$.

\subsubsection{Rough singular integrals}

Let us fix a dyadic lattice $\mathcal{D}$ and let $\mathcal{D}_{j}$
$j=1,\dots,3^{n}$ obtained using the $3^{n}$ dyadic lattices trick.
Now let
\[
G=\left\{ \frac{T_{\Omega}(fv)(x)}{v(x)}>1\right\} \setminus\left\{ M_{uv}^{\mathcal{F}}(f)>\frac{1}{2}\right\} 
\]
where $\mathcal{F=\bigcup}_{j=1}^{3^{n}}\mathcal{D}_{j}$ and assume
that $\|f\|_{L^{1}(uv)}=1$. Then it suffices to prove that 
\[
uv(G)\leq c_{n,p}[uv]_{A_{\infty}}[v]_{A_{1}}[u]_{A_{1}(v)}[v]_{A_{\infty}}\log\left(e+[uv]_{A_{\infty}}[v]_{A_{1}}\right)+\frac{1}{2}uv(G).
\]
Note that 
\[
uv(G)\leq\left|\int_{\mathbb{R}^{n}}\frac{T_{\Omega}(fv)}{v}uvg\right|=\left|\int_{\mathbb{R}^{n}}T_{\Omega}(fv)ug\right|
\]
where $g\simeq\chi_{G}$. Then for $s=1+\frac{1}{2\tau_{n}[uv]_{A_{\infty}}}$,
notice that, arguing as in \cite{LiPRR}
\[
\frac{(uv)^{s}(G\cap Q)}{(uv)^{s}(Q)}\lesssim\left(\frac{uv(G\cap Q)}{uv(Q)}\right)^{\frac{1}{2}}
\]
Taking that into account we have that
\[
\begin{split}uv(G) & \lesssim c_{n}c_{T}[uv]_{A_{\infty}}\sum_{j=1}^{3^{n}}\sum_{Q\in\mathcal{S}_{j}}\langle fv\rangle_{Q,1}\langle\chi_{G}u\rangle_{Q,s}|Q|\\
 & \leq c_{n}c_{T}[uv]_{A_{\infty}}\sum_{j=1}^{3^{n}}\sum_{Q\in\mathcal{S}_{j}}\langle f\rangle_{Q,1}^{v}\frac{v(Q)}{|Q|}\langle\chi_{G}u\rangle_{Q,s}|Q|\\
 & \leq c_{n}c_{T}[uv]_{A_{\infty}}[v]_{A_{1}}\sum_{j=1}^{3^{n}}\sum_{Q\in\mathcal{S}_{j}}\langle f\rangle_{Q,1}^{v}\langle\chi_{G}uv\rangle_{Q,s}|Q|\\
 & \leq c_{n}c_{T}[uv]_{A_{\infty}}[v]_{A_{1}}\sum_{j=1}^{3^{n}}\sum_{Q\in\mathcal{S}_{j}}\langle f\rangle_{Q,1}^{v}\langle\chi_{G}\rangle_{Q,s}^{(uv)^{s}}\left(\frac{(uv)^{s}(Q)}{|Q|}\right)^{\frac{1}{s}}\\
 & \leq c_{n}c_{T}[uv]_{A_{\infty}}[v]_{A_{1}}\sum_{j=1}^{3^{n}}\sum_{Q\in\mathcal{S}_{j}}\langle f\rangle_{Q,1}^{v}\left(\langle g\rangle_{Q,1}^{uv}\right)^{\frac{1}{2}}\left(\frac{(uv)^{s}(Q)}{|Q|}\right)^{\frac{1}{s}}|Q|\\
 & \leq c_{n}c_{T}[uv]_{A_{\infty}}[v]_{A_{1}}\sum_{j=1}^{3^{n}}\sum_{Q\in\mathcal{S}_{j}}\langle f\rangle_{Q,1}^{v}\left(\langle g\rangle_{Q,1}^{uv}\right)^{\frac{1}{2}}uv(Q).
\end{split}
\]
Hence it suffices to prove that for every sparse family $\mathcal{S}$,
\[
\begin{split} & c_{n}c_{T}[uv]_{A_{\infty}}[v]_{A_{1}}\sum_{Q\in\mathcal{S}}\langle f\rangle_{Q,1}^{v}\left(\langle g\rangle_{Q,1}^{uv}\right)^{\frac{1}{2}}uv(Q)\\
 & \leq c_{n,p}[uv]_{A_{\infty}}[v]_{A_{1}}[u]_{A_{1}(v)}[v]_{A_{\infty},}\log\left(e+[uv]_{A_{\infty}}[v]_{A_{1}}\right)+\frac{1}{2}uv(G).
\end{split}
\]
We split the sparse family as follows $Q\in\mathcal{S}_{k,j}$, $k,j\geq0$
if
\[
\begin{split}2^{-j-1} & <\langle f\rangle_{Q,1}^{v}\leq2^{-j}\\
2^{-k-1} & <\langle g\rangle_{Q,1}^{uv}\leq2^{-k}
\end{split}
\]
Let us call
\[
s_{k,j}=\sum_{Q\in\mathcal{S}_{k,j}}\langle f\rangle_{Q,1}^{v}\left(\langle g\rangle_{Q,1}^{uv}\right)^{\frac{1}{2}}uv(Q)
\]
Now we observe that 
\[
s_{k,j}\leq\begin{cases}
c_{n}2^{-k}[u]_{A_{1}(v)}[v]_{A_{\infty}}\\
c_{n}[uv]_{A_{\infty}}2^{-j}2^{k}uv\left(G\right).
\end{cases}
\]
For the top estimate we argue as we did to get \eqref{eq:TopThm2}.
 For the lower estimate, using Lemma \ref{Lem:UnQ}, 
\[
\begin{split}s_{k,j} & \leq2^{-j}2^{-k}\sum_{Q\in\mathcal{S}_{j,k}}uv(Q)\\
 & =c[uv]_{A_{\infty}}2^{-j}2^{-k}uv\left(\bigcup_{Q\in\mathcal{S}_{j,k}}Q\right)\\
 & \leq c_{n}[uv]_{A_{\infty}}2^{-j}2^{-k}uv\left(\left\{ x\in\mathbb{R}^{n}\,:\,(M_{uv}^{\mathcal{D}}g)^{\frac{1}{2}}>2^{-k-1}\right\} \right).
\end{split}
\]
Since $v\in A_{p}(u)$, taking into account Lemmas \ref{Lem:ContAvg}
and \ref{Lem:MUnDyad}, 
\[
\begin{split} & \leq c_{n}[uv]_{A_{\infty}}2^{-j}2^{-k}uv\left(\left\{ x\in\mathbb{R}^{n}\,:\,M_{uv}^{\mathcal{D}}g>2^{-2(k+1)}\right\} \right)\\
 & \leq c_{n}[uv]_{A_{\infty}}2^{-j}2^{-k}2^{2(k+1)}uv\left(G\right)\\
 & \leq c_{n}[uv]_{A_{\infty}}2^{-j}2^{k}uv\left(G\right).
\end{split}
\]
Combining the estimates above
\[
\begin{split}uv(G) & \leq c_{n}c_{T}\sum_{k=0}^{\infty}\sum_{j=0}^{\infty}s_{k,j}\\
 & \leq\sum_{k=0}^{\infty}\sum_{j=0}^{\infty}\min\left\{ c_{n}c_{T}[uv]_{A_{\infty}}[v]_{A_{1}}[u]_{A_{1}(v)}[v]_{A_{\infty}}2^{-k},c_{n}[uv]_{A_{\infty}}^{2}[v]_{A_{1}}2^{-j}2^{k}uv\left(G\right)\right\} .
\end{split}
\]
We end the proof using Lemma \ref{Lem:DoubleSum}, with 
\[
\gamma_{1}=c_{n}c_{T}[uv]_{A_{\infty}}[v]_{A_{1}}[u]_{A_{1}(v)}[v]_{A_{\infty},}\qquad\gamma_{2}=c_{n}[uv]_{A_{\infty}}^{2}[v]_{A_{1}},
\]
$\beta=uv\left(G\right)$, $\delta=1$ and $\gamma=3^{n}.$

\section*{Acknowledgment}

The authors would like to thank Sheldy Ombrosi for his comments on
an earlier version of this manuscript and for some enlightening discussions
on this topic. 

\bibliographystyle{plain}
\bibliography{referencias}

\end{document}